\documentclass[10pt]{amsart}
\usepackage{amsmath,amscd}
\usepackage{amsbsy}
\usepackage{amssymb}
\usepackage{amscd,amsthm}
\usepackage[all,cmtip]{xy}

\newtheorem{thm}{Theorem}\numberwithin{thm}{section}

\newtheorem{prop}[thm]{Proposition}
\newtheorem{cor}[thm]{Corollary}

\newtheorem{rema}[thm]{Remark}

\newtheorem{defi}[thm]{Definition}

\newtheorem*{thm2}{Theorem}

\newtheorem*{con2}{Conjecture}

\begin{document}
\begin{center}
\huge{No phantoms in the derived category of curves over arbitrary fields, and derived characterizations of Brauer--Severi varieties}\\[1cm]
\end{center}
\begin{center}

\large{Sa$\mathrm{\check{s}}$a Novakovi$\mathrm{\acute{c}}$}\\[0,4cm]
\small{\emph{To my wife Anastasia with thankfulness and love}}\\[0,6cm]
{\small August 2017}\\[0,3cm]
\end{center}

{\small \textbf{Abstract}.
In this paper we show that the derived category of Brauer--Severi curves satisfies the Jordan--H\"older property and cannot have quasi-phantoms, phantoms or universal phantoms. In this way we obtain that quasi-phantoms, phantoms or universal phantoms cannot exist in the derived category of smooth projective curves over a field $k$. Moreover, we show that a $n$-dimensional Brauer--Severi variety is completely characterized by the existence of a full weak exceptional collection consisting of pure vector bundles of length $n+1$, at least in characteristic zero. We conjecture that Brauer--Severi varieties $X$ satisfy $\mathrm{rdim}_{\mathrm{cat}}(X)=\mathrm{ind}(X)-1$, provided period equals index, and prove this in the case of curves, surfaces and for Brauer--Severi varieties of index at most three. We believe that the results for curves are known to the experts. We nevertheless give the proofs, adding to the literature.  
\begin{center}
\tableofcontents
\end{center}
\section{Introduction}
The bounded derived category $D^b(X)$ of a smooth projective variety has been recognized as an interesting invariant encoding a lot of geometric information. For instance, there are links between the semiorthogonal decomposition of $D^b(X)$ to the birational geometry of $X$ (see for instance \cite{KUZZ},\cite{KUZ1} and references therein). There are also links between the existence of special types of semiorthogonal decompositions of $D^b(X)$ and the existence of $k$-rational points (see \cite{AB}, \cite{NO2}, \cite{NOO} and references therein). Recently, special examples of semiorthogonal decompositions have been constructed. Namely, it was proved for several complex algebraic surfaces $S$ that $D^b(S)$ admits decompositions in which one component has trivial Hochschild homology and finite or trivial Grothendieck group (see \cite{AO}, \cite{BGS},\cite{BGKS},\cite{CL},\cite{GSH},\cite{GO} and \cite{LEE}). The first were called quasi-phantom whereas the latter were called phantoms. There is also the notion of universal phantom categories. There are several reasons one is interested in the (non-) existence of (quasi-) phantoms in $D^b(X)$ some of which are related to questions about noetherianity, rationality or the failure of the Jordan--H\"older property. Up to now it is still an open problem whether there is a phantom in $D^b(\mathbb{P}^2)$. We believe that the following result is known to the experts, but, to our best knowledge, stated nowhere. The prove we give is intrinsic and based on properties of central simple algebras.
\begin{thm}
Let $C$ be a smooth projective curve over a field $k$. Then $D^b(C)$ cannot have quasi-phantoms, phantoms or universal phantoms.
\end{thm}
For reasons related to several conjectures on central simple algebras, we are interested in derived characterizations of Brauer--Severi varieties and its invariants. Denote by $k^s$ a separable closure of the field $k$. Recall that a $k$-scheme $X$ is called a \emph{Brauer--Severi variety} if $X\otimes_k k^s \simeq \mathbb{P}^n_{k^s}$ for some $n$. We say $X$ is split if $X\simeq \mathbb{P}^n_k$. 
Via Galois cohomology, Brauer--Severi varieties are in one-to-one correspondence with central simple $k$-algebras. A finite dimensional associative $k$-algebra $A$ is called \emph{central simple} if the only two-sided ideals are $0$ and $A$ and whose center equals $k$. For any central simple $k$-algebra $A$ there is an integer $n>0$ and a division algebra $D$, such that $A\simeq M_n(D)$. The division algebra $D$ is also central and unique up to isomorphism. Note that a finite dimensional associative $k$-algebra $A$ is central simple if and only if $A\otimes_k k^s\simeq M_n(k^s)$. 
Two central simple $k$-algebras $A\simeq M_n(D)$ and $B\simeq M_m(D')$ are called \emph{Brauer-equivalent} if $D\simeq D'$. Recall that the \emph{Brauer group} $\mathrm{Br}(k)$ of a field $k$ is the group whose elements are equivalence classes of central simple $k$-algebras, with addition given by the tensor product of algebras. 
It is a fact that the Brauer group of any field is a torsion group. The order of an equivalence class $[A]\in \mathrm{Br}(k)$ is called the \emph{period} of $[A]$ and is denoted by $\mathrm{per}(A)$. The \emph{degree} of a central simple algebra $A=M_n(D)$ is defined to be $\mathrm{deg}(A):=\sqrt{\mathrm{dim}_k A}$ whereas the degree of the unique central division algebra $D$ is called the \emph{index} of $A$. Furthermore, the index of $A$ is divided by the period of $A$ and both have the same prime factors (see \cite{GS}, Proposition 4.5.13). 
The \emph{index} of a Brauer--Severi variety $X$ is defined to be the index of the corresponding central simple algebra and will be denoted by $\mathrm{ind}(X)$. It is also worth to mention that if the Brauer--Severi variety $X$ corresponds to $A$, one has $\mathrm{dim}(X)=\mathrm{deg}(A)-1$. For details on Brauer--Severi varieties, central simple algebras and their invariants we refer to \cite{AR} and \cite{GS}.


For a smooth projective variety $X$ over a field $k$, we denote by $D^b(X)$ the bounded derived category of coherent sheaves. 
We very briefly recall the definitions of weak exceptional collections and semiorthogonal decompositions and refer to \cite{BER1} and references therein. 

\textnormal{An object $\mathcal{E}^{\bullet}\in D^b(X)$ is called \emph{weak exceptional} (or w-exceptional for short) if $\mathrm{End}(\mathcal{E}^{\bullet})=A$, for some (not necessarily central) division $k$-algebra $A$ and $\mathrm{Hom}(\mathcal{E}^{\bullet},\mathcal{E}^{\bullet}[r])=0$ for $r\neq 0$. If $A=k$ the object is called \emph{exceptional}.}
\textnormal{A totally ordered set $\{\mathcal{E}^{\bullet}_1,...,\mathcal{E}^{\bullet}_n\}$ of w-exceptional objects on $X$ is called an \emph{w-exceptional collection} if $\mathrm{Hom}(\mathcal{E}^{\bullet}_i,\mathcal{E}^{\bullet}_j[r])=0$ for all integers $r$ whenever $i>j$. An w-exceptional collection is \emph{full} if $\langle\{\mathcal{E}^{\bullet}_1,...,\mathcal{E}^{\bullet}_n\}\rangle=D^b(X)$ and \emph{strong} if $\mathrm{Hom}(\mathcal{E}^{\bullet}_i,\mathcal{E}^{\bullet}_j[r])=0$ whenever $r\neq 0$. If the set $\{\mathcal{E}^{\bullet}_1,...,\mathcal{E}^{\bullet}_n\}$ consists of exceptional objects it is called \emph{exceptional collection}}. 

A generalization of the notion of a full w-exceptional collection is that of a semiorthogonal decomposition of $D^b(X)$. Recall that a full triangulated subcategory $\mathcal{A}$ of $D^b(X)$ is called \emph{admissible} if the inclusion $\mathcal{D}\hookrightarrow D^b(X)$ has a left and right adjoint functor.
\textnormal{Let $X$ be a smooth projective variety over $k$. A sequence $\mathcal{A}_1,...,\mathcal{A}_n$ of full triangulated subcategories of $D^b(X)$ is called \emph{semiorthogonal} if all $\mathcal{A}_i\subset D^b(X)$ are admissible and $\mathcal{A}_j\subset \mathcal{A}_i^{\perp}=\{\mathcal{F}^{\bullet}\in D^b(X)\mid \mathrm{Hom}(\mathcal{G}^{\bullet},\mathcal{F}^{\bullet})=0$, $\forall$ $ \mathcal{G}^{\bullet}\in\mathcal{A}_i\}$ for $i>j$. Such a sequence defines a \emph{semiorthogonal decomposition} of $D^b(X)$ if the smallest full triangulated subcategory containing all $\mathcal{A}_i$ equals $D^b(X)$.}
For a semiorthogonal decomposition we write $D^b(X)=\langle \mathcal{A}_1,...,\mathcal{A}_n\rangle$. If $\mathcal{A}$ is an admissible subcategory of $D^b(X)$, one has $D^b(X)=\langle \mathcal{A}^{\perp},\mathcal{A}\rangle$. 
If $D^b(X)$ admits a semiorthogonal decomposition $D^b(X)=\langle\mathcal{A}_1,...,\mathcal{A}_n\rangle$ one has
\begin{eqnarray*}
K_0(X)\simeq K_0(\mathcal{A}_1)\oplus...\oplus K_0(\mathcal{A}_n).
\end{eqnarray*}
\textnormal{It is easy to verify that if $\mathcal{E}^{\bullet}_1,...,\mathcal{E}^{\bullet}_n$ is a full w-exceptional collection on $X$, then by setting $\mathcal{A}_i=\langle\mathcal{E}^{\bullet}_i\rangle$ one gets a semiorthogonal decomposition $D^b(X)=\langle \mathcal{A}_1,...,\mathcal{A}_n\rangle$.}
In \cite{BER} Bernardara constructed a semiorthogonal decomposition for Brauer--Severi varieties. Let $X$ be the Brauer--Severi variety corresponding to a central simple $k$-algebra $A$. Then 
\begin{eqnarray}
D^b(X)=\langle D^b(k), D^b(A),..., D^b(A^{\otimes \mathrm{dim}(X)})\rangle
\end{eqnarray}
is a semiorthogonal decomposition of $D^b(X)$. For a wonderful and comprehensive overview of the theory on semiorthogonal decompositions and its relevance in algebraic geometry we refer to \cite{KU}.

By a theorem of Bondal and Orlov, a Brauer--Severi variety can be recovered from its derived category. So it is natural to study how birationality of two Brauer--Severi varieties is detected in their respective derived categories (see \cite{NOO},\cite{NO2}). In \cite{NO2} the author gives a derived characterization for a Brauer--Severi variety to be split, i.e. to be birational to $\mathbb{P}^n$. In this context, Theorem 1.2 below gives a characterizations of Brauer--Severi varieties in terms of their derived category. Let us briefly recall base change of semiorthogonal decompositions. If $T$ is a $k$-linear triangulated category with dg-enhancement and $K/k$ a field extension, we denote by $T_K$ the extension of scalars category defined in \cite{SO}. As expected, if $X$ is a smooth projective $k$-variety, $D^b(X)_K\simeq D^b(X\otimes_k K)$. Moreover, if $T$ is an admissible subcategory of $D^b(X)$, then $T_K$ is admissible in $D^b(X\otimes_k K)$ and one can show that $\langle A,B\rangle=D^b(X)$ if and only if $\langle A_K,B_K\rangle=D^b(X\otimes_k K)$. For a more general treatment of base change of semiorthogonal decompositions, we refer to \cite{KUZ}.
\begin{thm}
Let $X$ be a smooth projective variety of dimension $n$ over a field $k$ of characteristic zero. Then $X$ is a Brauer--Severi variety if and only if there is a semiorthogonal decomposition $D^b(X)=\langle \mathcal{A}_0,\mathcal{A}_1,...,\mathcal{A}_n\rangle$ such that its base change to $k^s$ is a semiorthogonal decomposition of the form $D^b(X_{k^s})=\langle (\mathcal{A}_0)_{k^s},(\mathcal{A}_1)_{k^s},...,(\mathcal{A}_n)_{k^s}\rangle$ with $(\mathcal{A}_i)_{k^s}\simeq \langle \mathcal{L}_i\rangle$ for some line bundles $\mathcal{L}_i$.
\end{thm}
 Recall from \cite{AE}, a vector bundle $\mathcal{E}$ on a proper $k$-variety $X$ is called \emph{pure of type} $\mathcal{L}$ if it splits after base change as $\mathcal{E}\otimes_k k^s\simeq \mathcal{L}^{\oplus m}$ for some line bundle $\mathcal{L}$ on $X\otimes_k k^s$. Throughout the work we call such bundles \emph{pure}.
\begin{cor}
Let $X$ be a smooth projective variety of dimension $n$ over a field $k$ of characteristic zero. Then $X$ is a Brauer--Severi variety if and only if there is a full w-exceptional collection $\mathcal{V}_0,...,\mathcal{V}_n$ consisting of pure vector bundles. The variety $X$ is a non-split Brauer--Severi variety if and only if there is full w-exceptional collection $\mathcal{V}_0,...,\mathcal{V}_n$ of pure vector bundles and $l\in\{0,...,n\}$ such that $\mathrm{ind}(\mathrm{End}(\mathcal{V}_l))>1$.
\end{cor}
Let $X$ be a smooth projective variety over $k$. We say $D^b(X)$ is \emph{representable in dimension $m$} if there is a semiorthogonal decomposition $D^b(X)=\langle \mathcal{A}_1,...,\mathcal{A}_n\rangle$ and for each $i=1,...,n$ there exists smooth projective connected varieties $Y_i$ with $\mathrm{dim}(Y_i)\leq m$, such that $\mathcal{A}_i$ is equivalent to an admissible subcategory of $D^b(Y_i)$ (see \cite{AB1} for details). We use the following notation
\begin{eqnarray*}
\mathrm{rdim}_{\mathrm{cat}}(X):=\mathrm{min}\{m\mid \textnormal{X is representable in dimension m}\},
\end{eqnarray*}
whenever such a finite $m$ exists. For Brauer--Severi varieties $X$ we observe $\mathrm{rdim}_{\mathrm{cat}}(X)\leq\mathrm{ind}(X)-1$ (see Proposition 4.1). We can prove even more, namely, we recover the index by the following result.
\begin{thm}
Let $X$ be a Brauer--Severi variety of index $\leq 3$. Then $\mathrm{rdim}_{\mathrm{cat}}(X)=\mathrm{ind}(X)-1$.
\end{thm}
In \cite{NO2} it is proved that a Brauer--Severi variety $X$ is split if and only if $\mathrm{rdim}_{\mathrm{cat}}(X)=0$. 
We formulate the following conjecture which gives a derived interpretation of the index, at least in the case period equals index.
\begin{con2}
Let $X$ be a Brauer--Severi variety with same period and index. Then $\mathrm{rdim}_{\mathrm{cat}}(X)=\mathrm{ind}(X)-1$.
\end{con2}
We want to mention that it is indeed a challenging problem to determine $\mathrm{rdim}_{\mathrm{cat}}(X)$ for a given Brauer--Severi variety or to find some kind of formula for $\mathrm{rdim}_{\mathrm{cat}}(X)$ depending on the invariants index and period.\\ 

{\small \textbf{Conventions}. Throughout this work $k$ denotes an arbitrary ground field, $k^s$ a separable and $\bar{k}$ an algebraic closure. Moreover, $D^b(X)$ denotes the bounded derived category of coherent sheaves on a smooth projective $k$-variety $X$.\\

{\small \textbf{Acknowledgement}. I wish to thank Marcello Bernardara and Pieter Belmans for very useful comments and the Heinrich--Heine--University for financial support via the SFF-grant. I also like to thank the referee for careful reading and suggestions which helped improve the paper. 

\section{Proof of Theorem 1.1}
Below we give the definitions of quasi-phantom and phantom subcategories as stated in \cite{GO}. For this, we need Hochschild homology of $D^b(X)$. We do not want to give the definition here and refer to \cite{KU0} for details, but we want to mention that they can be defined for admissible subcategories $\mathcal{A}\subset D^b(X)$. So let $\mathcal{D}$ be a triangulated category (for which Hochschild homology can be defined) and denote by $\mathrm{HH}_{\bullet}(\mathcal{D})$ its Hochschild homology. When $\mathcal{D}=D^b(X)$ is geometric, the Hochschild homology is defined by
\begin{eqnarray*}
\mathrm{HH}_{\bullet}(X):=\mathrm{HH}_{\bullet}(D^b(X)):=H^{\bullet}(X\times X, \Delta_*\mathcal{O}_X\otimes \Delta_*\mathcal{O}_X),
\end{eqnarray*}
where the push-forward and tensor product are meant to be derived.
If $D^b(X))=\langle \mathcal{A}_1,...,\mathcal{A}_m\rangle$ admits a semiorthogonal decomposition, one has 
\begin{eqnarray*}
\mathrm{HH}_{\bullet}(\mathcal{D})\simeq \bigoplus^m_{i=1}\mathrm{HH}_{\bullet}(\mathcal{A}_i).
\end{eqnarray*}

\begin{defi}
\textnormal{An admissible triangulated subcategory $\mathcal{A}$ of $D^b(X)$, where $X$ is a smooth projective variety will be called \emph{quasi-phantom} if $\mathrm{HH}_{\bullet}(\mathcal{A})=0$ and $K_0(\mathcal{A})$ is a finite abelian group. If in addition $K_0(\mathcal{A})=0$, it is called \emph{phantom}.}
\end{defi}
If $\mathcal{A}$ is admissible in $D^b(X)$ and $\mathcal{B}$ in $D^b(Y)$, we write $\mathcal{A}\boxtimes \mathcal{B}\subset D^b(X\times Y)$ for the smallest triangulated subcategory closed under taking direct summands and containing all objects of the form $p^*\mathcal{E}^{\bullet}\otimes q^*\mathcal{F}^{\bullet}$ (where $p$ and $q$ are the respective projections) for $\mathcal{E}^{\bullet}\in D^b(X)$ and $\mathcal{F}^{\bullet}\in D^b(Y)$.
\begin{defi}
\textnormal{We say that an admissible triangulated subcategory $\mathcal{A}$ of $D^b(X)$ is a \emph{universal phantom} if $\mathcal{A}\boxtimes D^b(Y)\subset D^b(X\times Y)$ is a phantom for any smooth projective variety $Y$.}
\end{defi}
To prove Theorem 1.1, we first show the following:
\begin{prop}
Let $X$ be the Brauer--Severi variety corresponding to a central simple algebra $A$. Assume $D^b(X)=\langle \mathcal{A}_0,\mathcal{A}_1,...,\mathcal{A}_n\rangle$ is a semiorthogonal decomposition, such that its base change to $k^s$ is a semiorthogonal decomposition of the form $D^b(X_{k^s})=\langle \mathcal{O}(p),\mathcal{O}(p+1),...,\mathcal{O}(p+n)\rangle$ for some $p\in \mathbb{Z}$. Then $\mathcal{A}_i\simeq D^b(A^{\otimes(p+i)})$. 
\end{prop}
\begin{proof}
Let $\pi\colon X_{k^s}\rightarrow X$ be the projection. By definition of the base change, $\mathcal{E}\in \mathcal{A}_i$ if and only if $\pi^*\mathcal{E}$ is a direct sum of shifts of $\mathcal{O}(p+i)$ in $D^b(X_{k^s})=D^b(\mathbb{P}^n)$, i.e. if and only if $\pi^*\mathcal{E}\in \langle\mathcal{O}(p+i)\rangle$. The same property have the components of Bernardara's semiorthogonal decomposition (1) from the introduction. Comparing $D^b(X)=\langle \mathcal{A}_0,\mathcal{A}_1,...,\mathcal{A}_n\rangle$ with the semiorthogonal decomposition (1), one concludes $\mathcal{A}_i\simeq D^b(A^{\otimes(p+i)})$. 
\end{proof}
\begin{defi}
\textnormal{The derived category $D^b(X)$ satisfies the \emph{noetherian property} if every increasing sequence $\mathcal{A}_1\subset \mathcal{A}_2\subset...$ of admissible subcategories becomes stationary. We say $D^b(X)$ has the \emph{Jordan--H\"older property} if $D^b(X)$ satisfies the noetherian property and if for any two maximal semiorthogonal decompositions $D^b(X)=\langle \mathcal{A}_1,...,\mathcal{A}_n\rangle$ and $D^b(X)=\langle \mathcal{B}_1,...,\mathcal{B}_m\rangle$ one has $n=m$ and there is a permutation $\sigma$ such that $\mathcal{B}_i\simeq \mathcal{A}_{\sigma(i)}$}.
\end{defi}
One can also define semiorthogonal decompositions for arbitrary $k$-linear triangulated categories $\mathcal{D}$ (see for instance \cite{KUZ1}). In this context one says that $\mathcal{D}$ is \emph{indecomposable} if it has no non-trivial semiorthogonal decomposition.
\begin{cor}
Let $C$ be the Brauer--Severi curve corresponding to a central simple algebra $A$. If $D^b(C)=\langle\mathcal{A}_0,...,\mathcal{A}_n\rangle$ is a semiorthogonal decomposition, then $n=1$ and there is a permutation $\sigma$ of $\{0,1,...,n\}$ with $\mathcal{A}_{\sigma(i)}\simeq D^b(A^{\otimes i})$, i.e., $D^b(C)$ satisfies the Jordan--H\"older property.
\end{cor}
\begin{proof}
Let $\mathcal{A}\subset D^b(C)$ be a non-trivial admissible subcategory. Then $D^b(C)=\langle \mathcal{A}^{\perp},\mathcal{A}\rangle$ is a semiorthogonal decomposition. After base change to $k^s$ we obtain a semiorthogonal decomposition 
\begin{eqnarray*}
D^b(\mathbb{P}^1)=D^b(C\otimes_k k^s)=\langle \bar{\mathcal{A}}^{\perp},\bar{\mathcal{A}}\rangle,
\end{eqnarray*}
where $\bar{\mathcal{A}}=\mathcal{A}\otimes_k k^s$ and $\bar{\mathcal{A}}^{\perp}=\mathcal{A}^{\perp}\otimes_k k^s$ are the triangulated subcategories obtained from base change (see \cite{KUZ}, Proposition 5.1 or alternatively \cite{TO}). Now one of the components of the latter semiorthogonal decomposition should contain an indecomposable sheaf of positive rank, i.e., a line bundle, say $\mathcal{O}(p)$. The other component should be contained in the left respectively right orthogonal of $\mathcal{O}(p)$, i.e., in $\langle \mathcal{O}(p-1)\rangle$ or in $\langle \mathcal{O}(p+1)\rangle$. Since the triangulated subcategories $\langle \mathcal{O}(p)\rangle$ and $\langle \mathcal{O}(p\pm 1)\rangle$ are indecomposable, we conclude that $\bar{\mathcal{A}}\simeq \langle \mathcal{O}(q)\rangle$ and $\bar{\mathcal{A}}^{\perp}\simeq \langle \mathcal{O}(q-1)\rangle$ for a suitable $q\in \mathbb{Z}$. By Proposition 2.3 it follows that $\mathcal{A}= D^b(A^{\otimes q})$ and $\mathcal{A}^{\perp}\simeq D^b(A^{\otimes (q-1)})$. Since we started with any admissible subcategory, it follows that each component in any semiorthogonal decomposition of $D^b(C)$ has this form. 
\end{proof}
\begin{cor}
Let $C$ be a Brauer--Severi curve. Then $D^b(C)$ cannot have quasi-phantoms, phantoms or universal phantoms.
\end{cor}
\begin{proof}
The proof of Corollary 2.5 shows that any admissible subcategory of $D^b(C)$ is of the form $D^b(A^{\otimes i})$ with $A$ being the central simple algebra corresponding to $C$. 
Since the Grothendieck group of the bounded derived category of a central simple $k$-algebra is not finite or trivial, $\mathcal{A}$ cannot be a quasi-phantom or phantom. Since a universal phantom is a phantom itself (just take $Y$ to be a point), the assertion follows.
\end{proof}


\begin{thm2}[Theorem 1.1]
Let $C$ be a smooth projective curve over a field $k$. Then $D^b(C)$ cannot not have quasi-phantoms, phantoms or universal phantoms.
\end{thm2}
\begin{proof}
It follows from \cite{KO}, Corollary 1.3 that the derived category of a smooth projective curve of genus $g>0$ does not have any non-trivial semiorthogonal decomposition. Together with Corollaries 2.5 and 2.6 this implies that the derived category of any smooth projective curve cannot have quasi-phantoms, phantoms or universal phantoms. This completes the proof of Theorem 1.1.
\end{proof}
\section{Proof of Theorem 1.2 and Corollary 1.3}
\begin{proof} (of Theorem 1.2)
Let $D^b(X)=\langle \mathcal{A}_0,\mathcal{A}_1,...,\mathcal{A}_n\rangle$ be a semiorthogonal decomposition such that the base change of $\langle \mathcal{A}_0,\mathcal{A}_1,...,\mathcal{A}_n\rangle$ to $k^s$ is a semiorthogonal decomposition of the form $D^b(X_{k^s})=\langle \mathcal{L}_0,\mathcal{L}_1,...,\mathcal{L}_n\rangle$, for some line bundles $\mathcal{L}_i$. Then $\{\mathcal{L}_0,\mathcal{L}_1,...,\mathcal{L}_n\}$ is a full exceptional collection on $X_{k^s}$. From \cite{V}, Theorem 1.2 we conclude $X\otimes_k k^s\simeq \mathbb{P}^n$ and therefore $X$ must be a Brauer--Severi variety. 

For the other implication, we notice that the semiorthogonal decomposition (1) base changes to $D^b(X_{k^s})=\langle \mathcal{O}(p),\mathcal{O}(p+1),...,\mathcal{O}(p+n)\rangle$ for some $p\in \mathbb{Z}$ by construction. This completes the proof. 
\end{proof}
Recall from the introduction that a vector bundle $\mathcal{E}$ on a proper $k$-variety $X$ is called pure if it splits after base change as $\mathcal{E}\otimes_k k^s\simeq \mathcal{L}^{\oplus m}$ for some line bundle $\mathcal{L}$ on $X\otimes_k k^s$
\begin{proof}(of Corollary 1.3)
Let $\mathcal{V}_0,...,\mathcal{V}_n$ be a full w-exceptional collection consisting of pure vector bundles on a $n$-dimensional variety $X$. Then $D^b(X)=\langle \mathcal{V}_0,...,\mathcal{V}_n\rangle$ is a semiorthogonal decomposition. By the definition of pure vector bundles we have after base change to $k^s$
\begin{eqnarray*}
\mathcal{V}_i\otimes_k k^s\simeq \mathcal{L}^{\oplus d_i}_i
\end{eqnarray*}
for some line bundle $\mathcal{L}_i$. Moreover, the isomorphisms
\begin{eqnarray*}
\mathrm{End}(\mathcal{V}_i)\otimes_k k^s\simeq \mathrm{End}(\mathcal{L}^{\oplus d_i}_i)\simeq M_{d_i}(k^s)
\end{eqnarray*}
imply that $\mathrm{End}(\mathcal{V}_i)$ are central simple $k$-algebras (see Introduction). Therefore, the decomposition $D^b(X)=\langle \mathcal{V}_0,...,\mathcal{V}_n\rangle$ base changes to $D^b(X_{k^s})=\langle \mathcal{L}_0,\mathcal{L}_1,...,\mathcal{L}_n\rangle$ for some line bundles $\mathcal{L}_i$. From Theorem 1.2 we know that $X$ must be a Brauer--Severi variety. For the other implication see \cite{ORL}, Example 1.17. 

For the second part of the corollary, let $\mathcal{V}_0,...,\mathcal{V}_n$ be a full w-exceptional collection consisting of pure vector bundles and a $l\in\{0,...,n\}$ with $\mathrm{ind}(\mathrm{End}(\mathcal{V}_l))>1$. This implies $\mathrm{End}(\mathcal{V}_l)$ is a non-split central simple $k$-algebra. The first part of the corollary shows that $X$ must be a Brauer--Severi variety. Now let $\mathcal{L}_l$ be the line bundle on $X\otimes_k k^s\simeq \mathbb{P}^n$ for which $\mathcal{V}_l\otimes_k k^s\simeq \mathcal{L}_l^{\oplus m_l}$. Since $\mathrm{End}(\mathcal{V}_l)$ is non-split, the line bundle $\mathcal{L}_l$ does not descend to $X$. But this implies that $X$ is non-split. Indeed, this can be seen as follows: See \cite{W}, Lemma 2.3 to conclude that
\begin{eqnarray*}
\mathrm{Pic}(X)\rightarrow \mathrm{Pic}(X\otimes_k k^s)\simeq \mathbb{Z},& \mathcal{L}\mapsto \mathcal{L}\otimes_k k^s
\end{eqnarray*}
is injective. If $X$ is split, the map is obviously surjective. On the other hand, if it is surjective we get $\mathcal{O}(1)\in\mathrm{Pic}(X)$, implying $X\simeq \mathbb{P}^n$ (see \cite{AR}). So $X$ splits if and only if the above map $\mathrm{Pic}(X)\rightarrow \mathrm{Pic}(X\otimes_k k^s)$ is bijective.

For the other implication assume $X$ is a non-split Brauer--Severi variety and let $A\simeq M_m(D)$ be the corresponding central simple algebra. Then there is a full w-exceptional collection consisting of pure vector bundles $\mathcal{V}_0,...,\mathcal{V}_n$ with $\mathrm{End}(\mathcal{V}_1)\simeq D$ (see \cite{BLU} or \cite{ORL}, Example 1.17). As $X$ is non-split, $\mathrm{ind}(A)=\mathrm{ind}(D)>1$.  
\end{proof}

\section{Proof of Theorem 1.4}
By definition, one has $\mathrm{rdim}_{\mathrm{cat}}(X)\leq \mathrm{dim}(X)$. For Brauer--Severi varieties $X$ we observe the following:
\begin{prop}
Let $X$ be a Brauer--Severi variety. Then $\mathrm{rdim}_{\mathrm{cat}}(X)\leq \mathrm{ind}(X)-1$.
\end{prop}
\begin{proof}
Let $A$ be the central simple algebra corresponding to $X$. By (1) we have a semiorthogonal decomposition
\begin{eqnarray*}
D^b(X)=\langle D^b(k), D^b(A),..., D^b(A^{\otimes \mathrm{dim}(X)})\rangle.
\end{eqnarray*}
According to the Wedderburn Theorem, the central simple algebra $A$ is isomorphic to $M_n(B)$ for a unique central division algebra $B$ and some $n>0$. 
Morita-equivalence gives us $D^b(A^{\otimes i})\simeq D^b(B^{\otimes i})$. 
Now let $Y_B$ be the Brauer--Severi variety corresponding to $B$. As $B$ is a division algebra, we have $\mathrm{dim}(Y_B)=\mathrm{ind}(A)-1=\mathrm{ind}(B)-1$. Note that (1) actually implies that $D^b(B^{\otimes i})$ is an admissible subcategory of $D^b(Y_B)$ for any $i$. So we immediately conclude 
\begin{eqnarray*}
\mathrm{rdim}_{\mathrm{cat}}(X)\leq \mathrm{dim}(Y_B)=\mathrm{ind}(X)-1.
\end{eqnarray*}
\end{proof}
Recall from the book \cite{GTA} that the category $\bf{dgcat}$ of all (small) dg categories and dg functors carries a Quillen model structure whose weak equivalences are Morita equivalences. Let us denote by $\mathrm{Hmo}$ the homotopy category hence obtained and by $\mathrm{Hmo}_0$ its additivization. Now to any small dg category $\mathcal{A}$ one can associate functorially its \emph{noncommutative motive} $U(\mathcal{A})$ which takes values in $\mathrm{Hmo}_0$. This functor $U\colon \bf{dgcat}\rightarrow \mathrm{Hmo}_0$ is proved to be the \emph{universal additive invariant}. An additive invariant is any functor $E\colon \bf{dgcat}$ $\rightarrow \mathcal{D}$ taking values in an additive category $\mathcal{D}$ such that
\begin{itemize}
      \item[\bf (i)] it sends derived Morita equivalences to isomorphisms,\\
			
      \item[\bf (ii)] for any pre-triangulated dg category $\mathcal{A}$ admitting full pre-triangulated dg subcategories $\mathcal{B}$ and $\mathcal{C}$ such that $H^0(\mathcal{A})=\langle H^0(\mathcal{B}), H^0(\mathcal{C})\rangle$ is a semiorthogonal decomposition, the morphism $E(\mathcal{B})\oplus E(\mathcal{C})\rightarrow E(\mathcal{A})$ induced by the inclusions is an isomorphism.
\end{itemize}			
For central simple $k$-algebras one has the following comparison theorem, which will be applied in the proof of Theorem 1.3.
\begin{thm}[\cite{TA}, Theorem 2.19]
Let $A_1,...,A_n$ and $B_1,...,B_m$ be central simple $k$-algebras, then the following are equivalent:
\begin{itemize}
     \item[\bf (i)] There is an isomorphism
			\begin{eqnarray*}
		\bigoplus^n_{i=1}U(A_i)\simeq \bigoplus^m_{j=1}U(B_j).
			\end{eqnarray*}
			
      \item[\bf (ii)] The equality $n=m$ holds and for all $1\leq i\leq n$ 
			\begin{eqnarray*}
			[B_i]=[A_{\sigma(i)}]\in \mathrm{Br}(k)
			\end{eqnarray*}
			for some permutation $\sigma$ of $\{1,...,n\}$.
			\end{itemize}
\end{thm}	

\begin{thm2}[Theorem 1.4]
Let $X$ be a Brauer--Severi variety with $\mathrm{ind}(X)\leq 3$. Then $\mathrm{rdim}_{\mathrm{cat}}(X)=\mathrm{ind}(X)-1$.
\end{thm2}
\begin{proof}
According to \cite{NO2} a Brauer--Severi variety $X$ is split if and only if $\mathrm{rdim}_{\mathrm{cat}}(X)=0$ if and only if $\mathrm{ind}(X)-1=0$. This covers the case $\mathrm{ind}(X)=1$. So it remains to prove the assertion for $r:=\mathrm{ind}(X)\in\{2,3\}$. For this, let $A$ be the central simple algebra corresponding to $X$. 
By Proposition 4.1 one has $\mathrm{rdim}_{\mathrm{cat}}(X)\leq \mathrm{ind}(X)-1=r-1$. Assume by contradiction that $\mathrm{rdim}_{\mathrm{cat}}(X)< r-1$. So for $r=2$ this means $\mathrm{rdim}_{\mathrm{cat}}(X)=0$ which gives a contradiction, as $A$ is by assumption non-split. For $n=3$, $\mathrm{rdim}_{\mathrm{cat}}(X)< r-1$ means $\mathrm{rdim}_{\mathrm{cat}}(X)\leq 1$. But $\mathrm{rdim}_{\mathrm{cat}}(X)=0$ gives a contradiction, since $X$ is non-split. Below we prove that $\mathrm{rdim}_{\mathrm{cat}}(X)=1$ also gives a contradiction. 

Now by \cite{AB1}, Proposition 6.1.6 and 6.1.10 we conclude that if $\mathrm{rdim}_{\mathrm{cat}}(X)=1$ there must be a semiorthogonal decomposition of $D^b(X)$ whose components are either $D^b(K)$, where $K/k$ is a (finite) separable extension, or $D^b(D)$, where $D$ is a central division algebra with $\mathrm{ind}(D)\in\{1,2\}$, or $D^b(C)$, where $C$ is a smooth $k$-curve of positive genus. Note that $D^b(C)$ cannot be present either because $K_0(X)$ is torsion free, or because $HH_1(X)=0$. Remembering the semiorthogonal decomposition (1), we see that the noncommutative motive $U(\mathrm{perf}(X))$ decomposes as
\begin{eqnarray}
U(\mathrm{perf}(X))&\simeq&\bigoplus^n_{i=1}U(K_i)\oplus \left(\bigoplus^m_{j=1}U(D_j)\right)\\
									 &\simeq&U(k)\oplus U(A)\oplus...\oplus U(A^{\otimes \mathrm{dim}(X)})
\end{eqnarray}
for suitable central division algebras $D_j$ with $\mathrm{ind}(D_j)\in\{1,2\}$ and suitable separable extensions $K_i$.
Note that $K_0(X)\simeq \mathbb{Z}^{\oplus (\mathrm{dim}(X)+1)}$ and therefore $n+m=\mathrm{dim}(X)+1$. After base change to $\bar{k}$ we also have $K_0(X_{\bar{k}})\simeq \mathbb{Z}^{\oplus (\mathrm{dim}(X)+1)}$. 
Now for any $D^b(K_i)$ we obtain after base change $D^b(K_i)_{\bar{k}}\simeq \prod^{r_i}_{q=1}D^b(\bar{k})$.
Hence $\sum^n_{i=1}{r_i}+m=\mathrm{dim}(X)+1$. But this implies $\sum^n_{i=1}{r_i}=n$ and therefore $r_i=1$. The above isomorphisms (2) and (3) then give
\begin{eqnarray*}
\bigoplus^n_{i=1}U(k)\oplus \left(\bigoplus^m_{j=1}U(D_j)\right) \simeq U(k)\oplus U(A)\oplus...\oplus U(A^{\otimes \mathrm{dim}(X)}).
\end{eqnarray*} 
Then by Theorem 2.19 we conclude that $A$ must be split or that $A$ must be Brauer-equivalent to $D_j$ for some $j$. This contradicts $r=3$ and completes the proof.

\end{proof}
\begin{cor}
Let $X$ be a Brauer--Severi variety of dimension $\leq 2$. Then $\mathrm{rdim}_{\mathrm{cat}}(X)=\mathrm{ind}(X)-1$.
\end{cor}

\begin{cor}
Let $X$ be a Brauer--Severi curve or surface. Then $X$ is non-split if and only if $\mathrm{rdim}_{\mathrm{cat}}(X)=\mathrm{dim}(X)$.
\end{cor}
\begin{proof}
Theorem 1.4 yields $\mathrm{rdim}_{\mathrm{cat}}(X)=\mathrm{ind}(X)-1$. If $X$ is non-split we have $\mathrm{ind}(X)-1=\mathrm{dim}(X)$. For the other implication assume $X$ is split. Then \cite{NO2}, Proposition 5.1 yields $\mathrm{rdim}_{\mathrm{cat}}(X)=0\neq \mathrm{dim}(X)$.
\end{proof}
\begin{rema}
\textnormal{It is worth to mention that $\mathrm{rdim}_{\mathrm{cat}}(X)=\mathrm{ind}(X)-1$ cannot hold in general. Counterexamples can be find for instance in \cite{AB}, table 3 on page 27. It is not yet clear whether there is indeed a numerical relation between $\mathrm{rdim}_{\mathrm{cat}}(X)$ and the index respectively the period.}
\end{rema}

{\small MATHEMATISCHES INSTITUT, HEINRICH--HEINE--UNIVERSIT\"AT 40225 D\"USSELDORF, GERMANY}\\
E-mail adress: novakovic@math.uni-duesseldorf.de

\end{document}